\documentclass[symmetry,article,submit,moreauthors,dvipdfm,A4,12pt,]{mdpi} 
\lastpage{x}
\doinum{10.3390/------}
\pubvolume{xx}
\pubyear{2014}
\history{Received: xx / Accepted: xx / Published: xx}

\usepackage{graphicx}
\usepackage{amsfonts}
\usepackage{amssymb}

\def\S{{\mathrm{S}}}

\def\D{{\mathrm{D}}}
\def\Z{{\mathbb{Z}}}

\def\fix{{\mathrm{fix}}}

\def\TSG{{\mathrm{TSG}}}

\def\Aut{{\mathrm{Aut}}}
\def\lk{{\mathrm{lk}}}

\newtheorem*{complete}{Complete Graph Theorem}

\newtheorem*{A4}{$\mathbf{A_4}$ Theorem}
\newtheorem*{A5}{$\mathbf{A_5}$ Theorem}
\newtheorem*{S4}{$\mathbf{S_4}$ Theorem}

\newtheorem*{auto-thm}{Finite Order Theorem}
\newtheorem*{N4Cycle}{4-Cycle Theorem}
\newtheorem*{subgroup}{Subgroup Theorem}
\newtheorem*{CnG}{Conway Gordon}
\newtheorem*{SFP}{Lemma 2}
\newtheorem*{OL}{Lemma 1}

\newtheorem{definition}{Definition}
\theoremstyle{mdpi}
\newcounter{thm}
\newcounter{ex}
\newcounter{re}
\Title{Topological Symmetry Groups of Small Complete Graphs}

\Author{Dwayne Chambers and Erica Flapan }

\address[1]{Pomona College, Claremont, CA 91711, USA}

\corres{eflapan@pomona.edu }

\abstract{Topological symmetry groups were originally introduced to study the symmetries of non-rigid molecules, but have since been used to study the symmetries of any graph embedded in $\mathbb{R}^3$.  In this paper, we determine for each complete graph $K_n$ with $n\leq 6$, what groups can occur as topological symmetry groups or orientation preserving topological symmetry groups of some embedding of the graph in $\mathbb{R}^3$.}



\begin{document}


\section{Introduction}

Molecular symmetries  are important in many areas of chemistry.  Symmetry is used in interpreting results in crystallography, spectroscopy, and quantum chemistry, as well as in analyzing the electron structure of a molecule.  Symmetry is also used in designing new pharmaceutical products. But what is meant by ``symmetry'' depends on the rigidity of the molecule in question.

For small molecules which are likely to be rigid, the group of rotations, reflections, and combinations of rotations and reflections, is an effective way of representing molecular symmetries.  This group is known as the {\it point group}, of the molecule because it fixes a point of $\mathbb{R}^3$.  However, some molecules can rotate around particular bonds, and large molecules can even be quite flexible.   DNA demonstrates how flexible a long molecule can be.  Even relatively small non-rigid molecules may have symmetries which are not included in the point group and can even be achiral as a result of their ability to rotate around particular bonds.  For example, the left and right sides of the biphenyl derivative illustrated in Figure~\ref{Mislow} rotate simultaneously, independent of the central part of the molecule.  Because of these rotating pieces, this molecule is achiral though it cannot be rigidly superimposed on its mirror form.  A detailed discussion of the achirality of this molecule can be found in \cite{book}.

\begin{figure}[h!]    
\begin{center}
\includegraphics[height=2.9cm]{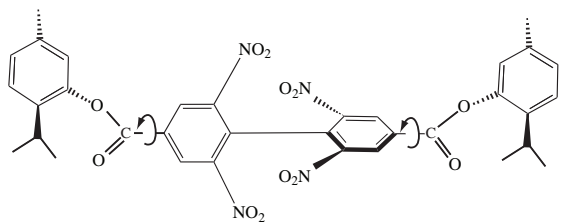}
\caption{The sides of this molecule rotate as indicated.}
\label{Mislow}
\end{center}
\end{figure}

 In general, the amount of rigidity of a given molecule depends on its chemistry not just its geometry.  Thus a purely mathematical definition of molecular symmetries that accurately reflects the behavior of all molecules is impossible.  However, for non-rigid molecules, a topological approach to classifying symmetries can add important information beyond what is obtained from the point group.  The definition of the topological symmetry group was first introduced by Jon Simon in 1987 in order to classify the symmetries of such molecules \cite{Si}.  More recently, topological symmetry groups have been used to study the symmetries of arbitrary graphs embedded in 3-dimensional space, whether or not such graphs represent molecules.  In fact, the study of symmetries of embedded graphs can be seen as a natural extension of the study of symmetries of knots which has a long history.  
 
 Though it may seem strange from the point of view of a chemist, the study of symmetries of embedded graphs as well as knots is more convenient to carry out in the 3-dimensional sphere $S^3=\mathbb{R}^3\cup \{\infty\}$ rather than in Euclidean 3-space, $\mathbb{R}^3$.  In particular, in $\mathbb{R}^3$ every rigid motion is a rotation, reflection, translation, or a combination of these operations.  Whereas, in $S^3$ glide rotations provide an additional type of rigid motion.  While a topological approach to the study of symmetries does not require us to focus on rigid motions, for the purpose of illustration it is preferable to display rigid motions rather than isotopies whenever possible.  Thus throughout the paper we work in $S^3$ rather than in $\mathbb{R}^3$.

\begin{definition} The {\bf topological symmetry group} of a graph $\Gamma$ embedded in $S^3$ is the subgroup of the automorphism group of the graph, $\Aut(\Gamma)$, induced by homeomorphisms of the pair $(S^3,\Gamma)$. The {\bf orientation preserving topological symmetry group}, $\mathrm{TSG}_+(\Gamma)$, is the subgroup of $\Aut(\Gamma)$ induced by orientation preserving homeomorphisms of $(S^3, \Gamma)$.  \end{definition}

It should be noted that for any homeomorphism $h$ of $(S^3,\Gamma)$, there is a homeomorphism $g$ of $(S^3,\Gamma)$ which fixes a point $p$ not on $\Gamma$ such that $g$ and $h$ induce the same automorphism on $\Gamma$.  By choosing $p$ to be the point at $\infty$, we can restrict $g$ to a homeomorphism of $(\mathbb{R}^3,\Gamma)$.  On the other hand if we start with an embedded graph $\Gamma$ in $\mathbb{R}^3$ and a homeomorphism $g$ of $(\mathbb{R}^3, \Gamma)$, we can consider $\Gamma$ to be embedded in $S^3=\mathbb{R}^3\cup \{\infty\}$ and extend $g$ to a homeomorphism of $S^3$ simply by fixing the point at $\infty$.  It follows that the topological symmetry group of $\Gamma$ in $S^3$ is the same as the topological symmetry group of $\Gamma$ in $\mathbb{R}^3$.  Thus we lose no information by working with graphs in $S^3$ rather than graphs in $\mathbb{R}^3$.

It was shown in \cite{TSG1} that the set of orientation preserving topological symmetry groups of 3-connected graphs embedded in $S^3$ is the same up to isomorphism as the set of finite subgroups of the group of orientation preserving diffeomorphisms of $S^3$, $\mathrm{Diff}_+(S^3)$.  However, even for a 3-connected embedded graph $\Gamma$, the automorphisms in $\mathrm{TSG}(\Gamma)$ are not necessarily induced by finite order homeomorphisms of $(S^3,\Gamma)$.  

For example, consider the embedded 3-connected graph $\Gamma$ illustrated in Figure~\ref{3connected}. The automorphism $(153426)$ is induced by a homeomorphism that slithers the graph along itself while interchanging the inner and outer knots in the graph.  On the other hand, the automorphism $(153426)$ cannot be induced by a finite order homeomorphism of $S^3$ because there is no order three homeomorphism of $S^3$ taking a figure eight knot to itself \cite{Ha, Tr} and the embedded graph in Figure~\ref{3connected} cannot be pointwise fixed by a finite order homeomorphism of $S^3$ \cite{Sm}.  

\begin{figure}[h!]    
\begin{center}
\includegraphics[height=2.9cm]{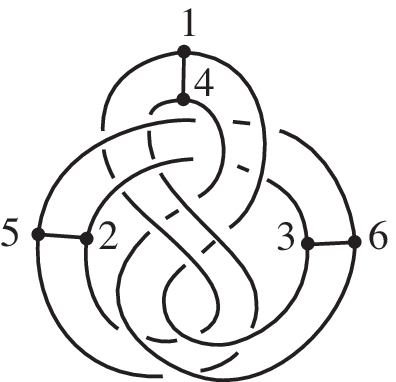}
\caption{The topological symmetry group of this embedded graph is not induced by a finite group of homeomorphisms of $S^3$.}
\label{3connected}
\end{center}
\end{figure}

On the other hand, Flapan proved the following theorem which we will make use of later in the paper.

\begin{auto-thm}
\cite{Rigidity} Let $\varphi$ be a non-trivial automorphism of a 3-connected graph $\gamma$ which is induced by a homeomorphism $h$ of $(S^3, \Gamma)$ for some embedding $\Gamma$ of $\gamma$ in $S^3$. Then for some  embedding $\Gamma'$ of $\gamma$ in $S^3$, the automorphism $\varphi$ is induced by a finite order homeomorphism, $f$ of $(S^3, \Gamma')$, and $f$ is orientation reversing if and only if $h$ is orientation reversing.
\end{auto-thm}

In the definition of the topological symmetry group, we start with a particular embedding $\Gamma$ of a graph $\gamma$ in $S^3$ and then determine the subgroup of the automorphism group of $\gamma$ which is induced by homeomorphisms of $(S^3,\Gamma)$.  However, sometimes it is more convenient to consider all possible subgroups of the automorphism group of an abstract graph, and ask which of these subgroups can be the topological symmetry group or orientation preserving topological symmetry group of some embedding of the graph in $S^3$. The following definition gives us the terminology to talk about topological symmetry groups from this point of view.

\begin{definition}
An automorphism $f$ of an abstract graph, $\gamma$, is said to be {\bf realizable} if there exists an embedding $\Gamma$ of $\gamma$ in $S^3$ such that $f$ is induced by a homeomorphism of $(S^3, \Gamma)$. A group $G$ is said to be {\bf realizable for $\mathbf{\gamma}$} if there exists an embedding $\Gamma$ of $\gamma$ in $S^3$ such that $\TSG(\Gamma) \cong G$.  If there exists an embedding $\Gamma$ such that $\TSG_+(\Gamma) \cong G$, then we say $G$ is {\bf positively realizable for $\gamma$}.
\end{definition}

It is natural to ask whether every finite group is realizable.  In fact, it was shown in \cite{TSG1} that the alternating group $A_m$ is realizable for some graph if and only if $m \leq 5$.  Furthermore, in \cite{TSG3} it was shown that for every closed, connected, orientable, irreducible 3-manifold $M$,
there exists an alternating group $A_m$ which is not isomorphic to the topological symmetry group of
 any graph embedded in $M$.


\section{\bf {Topological symmetry groups of complete graphs} }

For the special class of complete graphs $K_n$ embedded in $S^3$, Flapan, Naimi, and Tamvakis obtained the following  result.

\begin{complete}
\cite{TSG2} A finite group $H$ is isomorphic to $\TSG_+(\Gamma )$ for some embedding $\Gamma$ of a complete graph in $S^3$ if and only if $H$ is a finite subgroup of $\mathrm{SO}(3)$ or a subgroup of $\D_{m}\times \D_{m}$ for some odd $m$.
\end{complete}

However, this left open the question of what topological symmetry groups and orientation preserving topological symmetry groups are possible for embeddings of a particular complete graph $K_n$ in $S^3$. For each $n >6$, this question was answered for orientation preserving topological symmetry groups in the series of papers \cite{ChamFlap,POLY,SPA,CLAS}. 

In the current paper, we determine both the topological symmetry groups and the orientation preserving topological symmetry groups for all embeddings of $K_n$ in $S^3$ with $n\leq 6$.  Another way to state this is that we determine which groups are realizable and which groups are positively realizable for each $K_n$ with $n\leq 6$.  This is the first family of graphs for which both the realizable and the positively realizable groups have been determined.
 
 For $n\leq 3$, this question is easy to answer.  In particular, since $K_1$ is a single vertex, the only realizable or positively realizable group is the trivial group.   Since $K_2$ is a single edge, the only realizable or positively realizable group is $\mathbb{Z}_2$.  

For $n=3$, we know that $\mathrm{Aut}(K_3) \cong S_3\cong \D_3$, and hence every realizable or positively realizable group for $K_3$ must be a subgroup of $D_3$.   Note that for any embedding of $K_3$ in $S^3$, the graph can be ``slithered" along itself to obtain an automorphism of order 3 which is induced by an orientation preserving homeomorphism.  Thus the topological symmetry group and orientation preserving topological symmetry group of any embedding of $K_3$ will contain an element of order 3.  Thus neither the trivial group nor $\Z_2$ is realizable or positively realizable for $K_3$.  If $\Gamma$ is a planar embedding of $K_3$ in $S^3$, then $\TSG(\Gamma) =\TSG_+(\Gamma)\cong \D_3$.  Recall that the trefoil knot $3_1$ is chiral while the knot $8_{17}$ is negative achiral and non-invertible.  Thus if $\Gamma$ is the knot $8_{17}$, then no orientation preserving homeomorphism of $(S^3,\Gamma)$ inverts $\Gamma$, but there is an orientation reversing homeomorphism of $(S^3,\Gamma)$ which inverts $\Gamma$. Whereas, if $\Gamma$ is the knot $8_{17} \# 3_1$, then there is no homeomorphism of $(S^3, \Gamma)$ which inverts $\Gamma$.  Table 1 summarizes our results for $K_3$.

\begin{table}[h!b!p!]
\centering
\begin{tabular}{| p{4cm} | p{3cm} | p{3cm} |}
\hline
{\bf Embedding} & $\TSG(\Gamma)$ & $\TSG_+(\Gamma)$ \\ \hline
Planar & $\D_3$ & $\D_3$ \\ \hline
$8_{17}$ & $\D_3$ & $\Z_3$ \\ \hline
$8_{17}$ \# $3_1$ & $\Z_3$ & $\Z_3$ \\ \hline
\end{tabular}

\medskip

\caption{Realizable and positively realizable groups for $K_3$}
\label{table1}
\end{table}

Determining which groups are realizable and positively realizable for $K_4$, $K_5$, and $K_6$ is the main point of this paper.   In each case, we will first determine the positively realizable groups and then use the fact that either $\TSG_+(\Gamma) = \TSG(\Gamma)$ or $\TSG_+(\Gamma)$ is a normal subgroup of $\TSG(\Gamma)$ of index 2 to help us determine the realizable groups.  
\bigskip

In addition to the Complete Graph Theorem given above, we will make use of the following results in our analysis of positively realizable groups for $K_n$ with $n\geq 4$.

\begin{A4}
\cite{POLY} A complete graph $K_m$ with $m \geq 4$ has an embedding $\Gamma$ in $S^3$ such that $\TSG_+(\Gamma)\cong A_4$ if and only if $m \equiv 0$, 1, 4, 5, 8 (mod 12).
\end{A4}

\begin{A5}
\cite{POLY} A complete graph $K_m$ with $m \geq 4$ has an embedding $\Gamma$ in $S^3$ such that $\TSG_+(\Gamma)\cong A_5$ if and only if $m \equiv 0$, 1, 5, 20 (mod 60).
\end{A5}

\begin{S4}
\cite{POLY} A complete graph $K_m$ with $m \geq 4$ has an embedding $\Gamma$ in $S^3$ such that $\TSG_+(\Gamma)\cong \S_4$ if and only if $m \equiv 0$, 4, 8, 12, 20 (mod 24).
\end{S4}

\begin{subgroup} \cite{SPA}
Let $\Gamma$ be an embedding of a 3-connected graph $\gamma$ in $S^3$ with an edge that is not pointwise fixed by any non-trivial element of $\TSG_+(\Gamma)$.  Then every subgroup of $\TSG_+(\Gamma)$ is positively realizable for $\gamma$.
\end{subgroup}

\medskip

It was shown in \cite{SPA} that adding a local knot to an edge of a 3-connected graph is well-defined and that any homeomorphism of $S^3$ taking the graph to itself must take an edge with a given knot to an edge with the same knot. Furthermore, any orientation preserving homeomorphism of $S^3$ taking the graph to itself must take an edge with a given non-invertible knot to an edge with the same knot oriented in the same way.  Thus for $n>3$, adding a distinct knot to each edge of an embedding of $K_n$ in $S^3$ will create an embedding $\Delta$ where $\TSG(\Delta)$ and $\TSG_+(\Delta)$ are both trivial. Hence we do not include the trivial group in our list of realizable and positively realizable groups for $K_n$ when $n>3$.

Finally, observe that for $n>3$, for a given embedding $\Gamma$ of $K_n$ we can add identical chiral knots (whose mirror image do not occur in $\Gamma$) to every edge of $\Gamma$ to get an embedding $\Gamma'$ such that $\TSG(\Gamma') = \TSG_+(\Gamma)$. Thus every group which is positively realizable for $K_n$ is also realizable for $K_n$.  We will use this observation in the rest of our analysis.

\bigskip

\section*{\bf{Topological Symmetry Groups of $\mathbf{K_4}$}}

The following is a complete list of all the non-trivial subgroups of $\Aut(K_4) \cong \S_4$ up to isomorphism:
 $\S_4$, $A_4$, $\D_4$, $\D_3$, $\D_2$, $\Z_4$, $\Z_3$, $\Z_2$.

\begin{figure}[h!]
\begin{center}
\includegraphics[height=2.9cm]{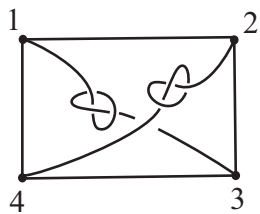}
\caption{$\TSG_+(\Gamma) \cong \D_4$.}
\label{K4D4}
\end{center}
\end{figure}

We will show that all of these groups are positively realizable, and hence all of the groups will also be realizable.  First consider the embedding $\Gamma$ of $K_4$ illustrated in Figure \ref{K4D4}. The square $\overline{1234}$ must go to itself under any homeomorphism of $(S^3, \Gamma)$.  Hence $\TSG_+(\Gamma)$ is a subgroup of $\D_4$.  In order to obtain the automorphism $(1234)$, we rotate the square $\overline{1234}$ clockwise by $90^\circ$ and pull $\overline{24}$ under $\overline{13}$.  We can obtain the transposition $(13)$ by first rotating the figure by $180^\circ$ about the axis which contains vertices 2 and 4 and then pulling $\overline{13}$ under $\overline{24}$.  Thus $\TSG_+(\Gamma)\cong\D_4$. Furthermore, since the edge $\overline{12}$ is not pointwise fixed by any non-trivial element of $\TSG_+(\Gamma)$, by the Subgroup Theorem the groups $\Z_4$, $\D_2$ and $\Z_2$ are each positively realizable for $K_4$.

Next, consider the embedding, $\Gamma$ of $K_4$ illustrated in Figure \ref{k4d3}. All homeomorphisms of $(S^3, \Gamma)$ fix vertex 4.  Hence $\TSG_+(\Gamma)$ is a subgroup of $D_3$.  The automorphism $(123)$ is induced by a rotation, and the automorphism $(12)$ is induced by turning the figure upside down and then pushing vertex 4 back up through the centre of $\overline{123}$.  Thus $\TSG_+(\Gamma) \cong \D_3$. Since the edge $\overline{12}$ is not pointwise fixed by any non-trivial element of $\TSG_+(\Gamma)$, by the Subgroup Theorem, the group $\Z_3$ is also positively realizable for $K_4$.

\begin{figure}[h!]
\begin{center}
\includegraphics[height=4.2cm]{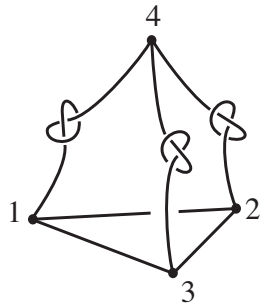}
\caption{$\TSG_+(\Gamma) \cong \D_3$.}
\label{k4d3}
\end{center}
\end{figure}

Thus every subgroup of $\Aut(K_4)$ is positively realizable.  Now by adding appropriate equivalent chiral knots to each edge, all subgroups of $\Aut(K_4)$ are also realizable.   We summarize our results for $K_4$ in Table~\ref{table2}.

\begin{table}[h!]
\centering
\begin{tabular}{| l | l | p{5cm} |}
\hline
{\bf Subgroup} & {\bf Realizable/Positively Realizable} & {\bf Reason} \\ \hline
$\S_4$ & Yes & By $\S_4$ Theorem \\ \hline
$A_4$ & Yes & By $A_4$ Theorem \\ \hline
$\D_4$ & Yes & By Figure \ref{K4D4} \\ \hline
$\D_3$ & Yes & By Figure \ref{k4d3}  \\ \hline
$\D_2$ & Yes & By Subgroup Theorem \\ \hline
$\Z_4$ & Yes & By Subgroup Theorem \\ \hline
$\Z_3$ & Yes & By Subgroup Theorem \\ \hline
$\Z_2$ & Yes & By Subgroup Theorem \\ \hline
\end{tabular}

\caption{Non-trivial realizable and positively realizable groups for $K_4$}
\label{table2}
\end{table}

\section{\bf{Topological Symmetry Groups of $\mathbf{K_5}$}}

The following is a complete list of all the non-trivial subgroups of $\Aut(K_5) \cong \S_5$:

\noindent $\S_5$, $A_5$, $\S_4$, $A_4$, $\Z_5 \rtimes \Z_4$, $\D_6$, $\D_5$, $\D_4$, $\D_3$, $\D_2$, $\Z_6$, $\Z_5$, $\Z_4$, $\Z_3$, $\Z_2$  (see \cite{gotz5} and  \cite{wiki1}).

The lemma below follows immediately from the Finite Order Theorem \cite{Rigidity} (stated in the introduction) together with Smith Theory \cite{Sm}.

\begin{OL}
Let $n > 3$ and let $\varphi$ be a non-trivial automorphism of $K_n$ which is induced by a homeomorphism $h$  of $(S^3, \Gamma)$ for some embedding $\Gamma$ of $K_n$ in $S^3$. If $h$ is orientation reversing, then $\varphi$ fixes at most 4 vertices. If $h$ is orientation preserving, then $\varphi$ fixes at most 3 vertices, and if $\varphi$ has even order, then $\varphi$ fixes at most 2 vertices. 
\end{OL}

We now prove the following lemma.

\begin{SFP}
Let $n > 3$ and let $\Gamma$ be an embedding of $K_n$ in $S^3$ such that $\TSG_+(\Gamma)$ contains an element $\varphi$ of even order $m>2$. Then $\varphi$ does not fix any vertex or interchange any pair of vertices. \end{SFP}

\begin{proof}  By the Finite Order Theorem, $K_n$ can be re-embedded as $\Gamma'$ so that $\varphi$ is induced on $\Gamma'$ by a finite order orientation preserving homeomorphism $h$.  Suppose that $\varphi$ fixes a vertex or interchanges a pair of vertices of $\Gamma'$. Then $\fix(h)$ is non-empty, and hence by Smith Theory, $\fix(h) \cong S^1$.  Let $r=m/2$.  Then $h^r$ induces an involution on the vertices of $\Gamma'$, and this involution can be written as a product $(a_1b_1)\dots(a_qb_q)$ of disjoint transpositions of vertices.  Now for each $i$, $h^r$ fixes a point on the edge $\overline{a_ib_i}$. But $\fix(h^r)$ contains $\fix(h)$ and thus by Smith Theory $\fix(h^r)=\fix(h)$. Hence $h$ fixes a point on each edge $\overline{a_ib_i}$. Thus $h$ induces also $(a_1b_1)\dots(a_qb_q)$ on the vertices of $\Gamma'$.   But this contradicts the hypothesis that the order of $\varphi$ is $m>2$.
\end{proof}

\medskip

By Lemma 2, there is no embedding of $K_5$ in $S^3$ such that $\TSG_+(\Gamma)$ contains an element of order 4 or of order 6. It follows that $\TSG_+(\Gamma)$ cannot be $\D_6$, $\Z_6$, $\D_4$ or $\Z_4$.

\begin{figure}[h!]
\begin{center}
\includegraphics[height=3.8cm]{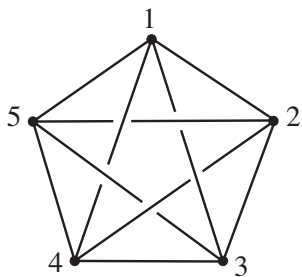}
\caption{$\TSG_+(\Gamma) \cong \D_5$.}
\label{K5D5}
\end{center}
\end{figure}

Consider the embedding $\Gamma$ of $K_5$ illustrated in Figure \ref{K5D5}.  The knotted cycle $\overline{13524}$ must be setwise invariant under every homeomorphism of $\Gamma$.  Thus $\TSG_+(\Gamma)\leq D_5$.  The automorphism $(12345)$ is induced by rotating $\Gamma$, and $(25)(34)$ is induced by turning the graph over. Hence $\TSG_+(\Gamma) = \langle (12345), (25)(34) \rangle \cong \D_5$.  Since the edge $\overline{12}$ is not pointwise fixed by any non-trivial element of $\TSG_+(\Gamma)$, by the Subgroup Theorem the groups $\Z_5$ and $\Z_2$ are also positively realizable for $K_5$.

Next consider the embedding $\Gamma$ of $K_5$ illustrated in Figure \ref{K5Z3}.  The triangle $\overline{123}$ must go to itself under any homeomorphism. Also by Lemma 1, any orientation preserving homeomorphism which fixes vertices 1, 2, and 3 induces a trivial automorphism on $K_5$. Thus $\TSG_+(\Gamma) \leq \D_3$.  The automorphism $(123)$ is induced by a rotation.  Also the automorphism $(45)(12)$ is induced by pulling vertex 4 down through the centre of triangle $\overline{123}$ while pulling vertex 5 into the centre of the figure then rotating by $180^\circ$ about the line through vertex 3 and the midpoint of the edge $\overline{12}$. Thus $\TSG_+(\Gamma) = \langle (123), (45)(12) \rangle \cong \D_3$.  Since the edge $\overline{12}$ is not pointwise fixed by any non-trivial element of $\TSG_+(\Gamma)$, by the Subgroup Theorem, the group $\Z_3$ is positively realizable for $K_5$.

\begin{figure}[h!]
\begin{center}
\includegraphics[height=3.9cm]{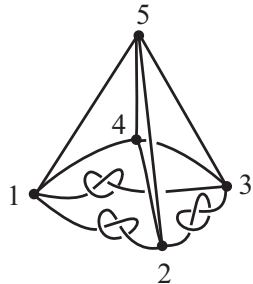}
\caption{$\TSG_+(\Gamma) \cong \D_3$.}
\label{K5Z3}
\end{center}
\end{figure}

Lastly, consider the embedding $\Gamma$ of $K_5$ illustrated in Figure \ref{K5D2} with vertex 5 at infinity. The square $\overline{1234}$ must go to itself under any homeomorphism. Hence $\TSG_+(\Gamma) \leq \D_4$.  The automorphism $(13)(24)$ is induced by rotating the square by $180^\circ$.  By turning over the figure we obtain $(12)(34)$. By Lemma 2, $\TSG_+(\Gamma)$ cannot contain an element of order 4. Thus $\TSG_+(\Gamma) = \langle (13)(24), (12)(34) \rangle \cong \D_2$.

\begin{figure}[h!]
\begin{center}
\includegraphics[height=3.3cm]{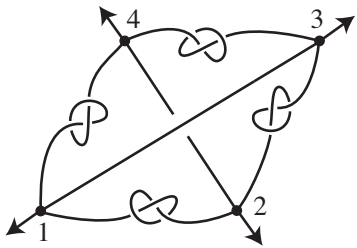}
\caption{$\TSG_+(\Gamma) \cong \D_2$.}
\label{K5D2}
\end{center}
\end{figure}

We summarize our results on positive realizability for $K_5$ in Table \ref{table3}. 

\begin{table}[h!]
\centering
\begin{tabular}{| l | l | p{7cm} |}
\hline
{\bf Subgroup} & {\bf Positively Realizable} & {\bf Reason} \\ \hline
$A_5$ & Yes & By $A_5$ Theorem \\ \hline
$\S_4$ & No & By $\S_4$ Theorem \\ \hline
$A_4$ & Yes & By $A_4$ Theorem \\ \hline
$\D_6$ & No & By Lemma 2\\ \hline
$\D_5$ & Yes & By Figure \ref{K5D5}\\ \hline
$\D_4$ & No & By Lemma 2\\ \hline
$\D_3$ & Yes & By Figure \ref{K5Z3}  \\ \hline
$\D_2$ & Yes & By Figure \ref{K5D2}  \\ \hline
$\Z_6$ & No & By Lemma 2\\ \hline
$\Z_5$ & Yes & By Subgroup Theorem \\ \hline
$\Z_4$ & No & By Lemma 2\\ \hline
$\Z_3$ & Yes & By Subgroup Theorem \\ \hline
$\Z_2$ & Yes & By Subgroup Theorem \\ \hline
\hline
\end{tabular}

\caption{Non-trivial positively realizable groups for $K_5$}
\label{table3}
\end{table}

Again by adding appropriate equivalent chiral knots to each edge, all of the positively realizable groups for $K_5$ are also realizable. Thus we only need to determine realizability for the groups $\S_5$, $\S_4$, $\Z_5 \rtimes \Z_4$, $\D_6$, $\D_4$, $\Z_6$, and $\Z_4$.

Let $\Gamma$ be the embedding of $K_5$ illustrated in Figure \ref{S5}. Any transposition which fixes vertex 5 is induced by a reflection through the plane containing the three vertices fixed by the transposition. To see that any transposition involving vertex 5 can be achieved, consider the automorphism $(15)$. Pull $\overline{51}$ through the triangle $\overline{234}$ and then turn over the embedding so that vertex $5$ is at the top, vertex 1 is in the centre and vertices 3 and 4 are switched.  Now reflect in the plane containing vertices 1, 5, and 2 in order to switch vertices 3 and 4 back.  All other transpositions involving vertex 5 can be induced by a similar sequence of moves. Hence $\TSG(\Gamma) \cong \S_5$.

\begin{figure}[h!]
\begin{center}
\includegraphics[height=3.3cm]{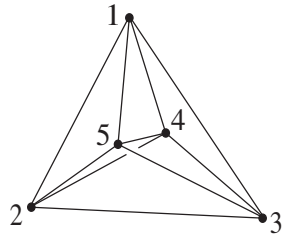}
\caption{$\TSG(\Gamma) \cong \S_5$.}
\label{S5}
\end{center}
\end{figure}

 We create a new embedding $\Gamma'$ from Figure \ref{S5} by adding the achiral figure eight knot, $4_1$, to all edges containing vertex 5. Now every homeomorphism of $(S^3, \Gamma')$ fixes vertex 5, yet all transpositions fixing vertex 5 are still possible.  Thus $\TSG(\Gamma') \cong \S_4$.

In order to prove $\D_4$ is realizable for $K_5$ consider the embedding $\Gamma$ illustrated in Figure \ref{K5D4}. Every homeomorphism of $(S^3, \Gamma)$ takes $\overline{1234}$ to itself, so $\TSG(\Gamma) \leq \D_4$. The automorphism $(1 2 34)$ is induced by rotating the graph by $90^{\circ}$ about a vertical line through vertex $5$,  then reflecting in the plane containing the vertices $1, 2, 3, 4$, and finally isotoping the knots into position.  Furthermore, reflecting in the plane containing $\overline{153}$ or $\overline{2 5 4}$ and then isotoping the knots into position yields the transposition $(2 4)$ or $(1 3)$ respectively. Hence $\TSG(\Gamma) \cong \D_4$.

\begin{figure}[h!]
\begin{center}
\includegraphics[height=4.3cm]{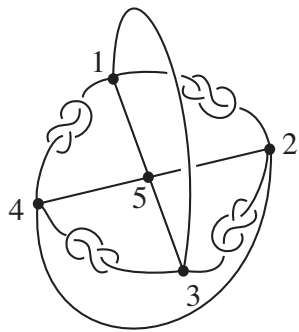}
\caption{$\TSG(\Gamma) \cong \D_4$.}
\label{K5D4}
\end{center}
\end{figure}

We obtain a new embedding $\Gamma'$ by replacing the invertible $4_1$ knots in Figure \ref{K5D4} with the knot $12_{427}$, which is positive achiral but non-invertible \cite{KnotS}. Since $12_{427}$ is neither negative achiral nor invertible, no homeomorphism of $(S^3,\Gamma')$ can invert $\overline{1234}$. Thus $\TSG(\Gamma') \cong \Z_4$.

Next let $\Gamma$ denote the embedding of $K_5$ illustrated in Figure \ref{K5D6}.  Every homeomorphism of $(S^3, \Gamma)$ takes $\overline{123}$ to itself, so $\TSG(\Gamma) \leq \D_6$.  The 3-cycle $(123)$ is induced by a rotation.  Each transposition involving only vertices 1, 2, and 3 is induced by a reflection in the plane containing $\overline{45}$ and the remaining fixed vertex followed by an isotopy.  The transposition $(4 5)$ is induced by a reflection in the plane containing vertices 1, 2 and 3 followed by an isotopy. Thus $\TSG(\Gamma) \cong \D_6$, generated by $(1 2 3)$, $(23)$, and $(4 5)$.

\begin{figure}[h!]
\begin{center}
\includegraphics[height=4.8cm]{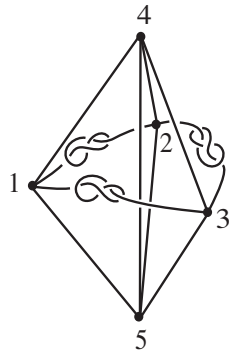}
\caption{$\TSG(\Gamma) \cong \D_6$.}
\label{K5D6}
\end{center}
\end{figure}

We obtain a new embedding $\Gamma'$ by replacing the $4_1$ knots in Figure \ref{K5D6} by $12_{427}$ knots.  Then the triangle $\overline{123}$ cannot be inverted.  Thus $\TSG(\Gamma') \cong \Z_6$, generated by $(1 2 3)$ and $(4 5)$.

It is more difficult to show that $\Z_5 \rtimes \Z_4$ is realizable for $K_5$, so we define our embedding in two steps.  First we create an embedding $\Gamma$ of $K_5$ on the surface of a torus $T$ that is standardly embedded in $S^3$. In Figure $\ref{FIG1}$, we illustrate $\Gamma$ on a flat torus.  Let $f$ denote a glide rotation of $S^3$ which rotates the torus longitudinally by $4\pi/5$ and while rotating it meridinally by $8\pi/5$. Thus $f$ takes $\Gamma$ to itself inducing the automorphism $(12345)$.

\begin{figure}[h!]
\begin{center}
\includegraphics[height=5cm]{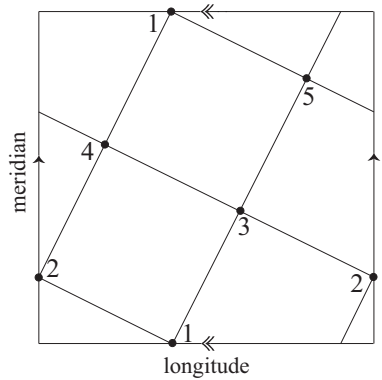}
\caption{The embedding $\Gamma$ of $K_5$ in a torus.}
\label{FIG1}
\end{center}
\end{figure}

Let $g$ denote the homeomorphism obtained by rotating $S^3$ about a $(1,1)$ curve on the torus $T$, followed by a reflection through a sphere meeting $T$ in two longitudes, and then a meridional rotation of $T$ by $6\pi/5$.  In Figure $\ref{FIG2}$, we illustrate the step-by-step action of $g$ on $T$, showing that $g$ takes $\Gamma$ to itself inducing $(2431)$.

\begin{figure}[h!]
\begin{center}
\includegraphics[height=4.7cm]{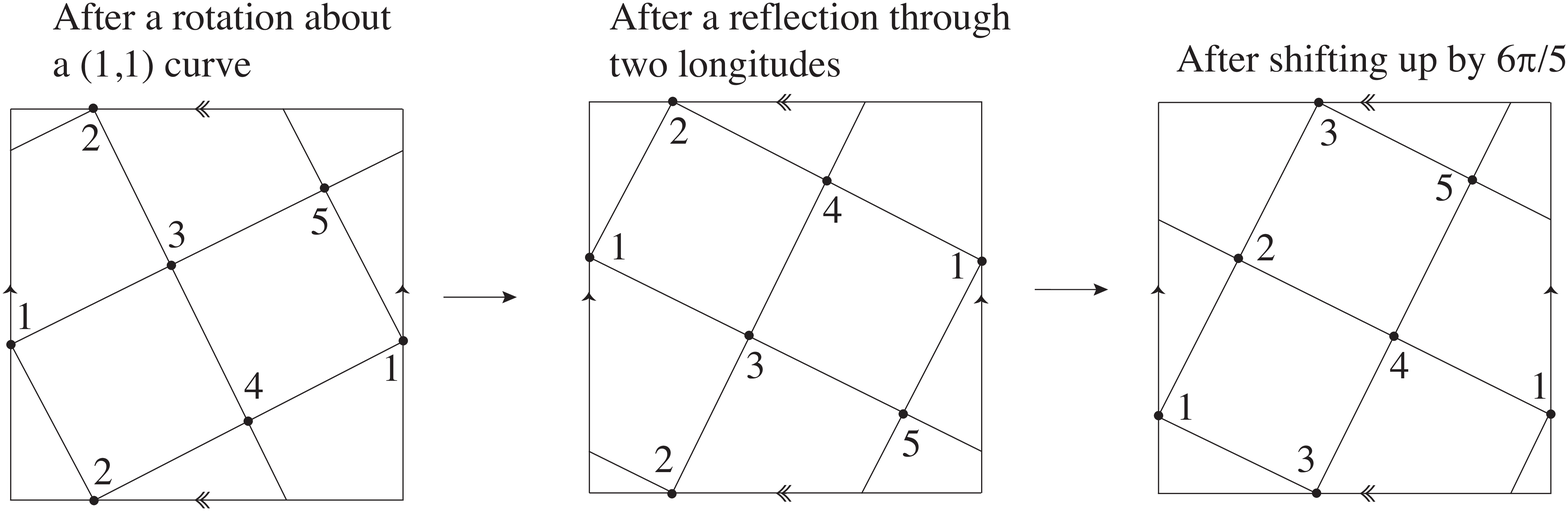}
\caption{The action of $g$ on $\Gamma$}
\label{FIG2}
\end{center}
\end{figure}

The homeomorphisms $f$ and $g$ induce the automorphisms $\phi=(1 2 3 4 5)$ and $\psi=(2 4 3 1)$ respectively.  Observe that $\phi^5 =\psi^4 =1$ and $\psi \phi=\phi \psi^2$. Thus $\langle \phi, \psi \rangle \cong\Z_5 \rtimes \Z_4 \leq \TSG(\Gamma) \leq \S_5$. Note however that the embedding in Figure $\ref{FIG1}$ is isotopic to the embedding of $K_5$ in Figure $\ref{S5}$. Thus $\TSG(\Gamma) \cong \S_5$.

In order to obtain the group $\Z_5 \rtimes \Z_4$, we now consider the embedding $\Gamma'$ of $K_5$ whose projection on a torus is illustrated in Figure \ref{TorFig8}. Observe that the projection of $\Gamma'$ in every square of the grid given by $\Gamma$ on the torus is identical.  Thus the homeomorphism $f$ which took $\Gamma$ to itself inducing the automorphism $\phi=(12345)$ on $\Gamma$ also takes $\Gamma'$ to itself inducing $\phi$ on $\Gamma'$.

\begin{figure}[h!]
\begin{center}
\includegraphics[height=5cm]{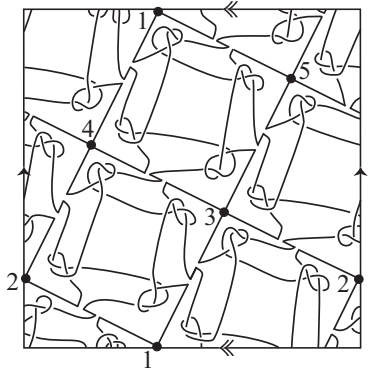}
\caption{Projection of $\Gamma'$ on the torus.}
\label{TorFig8}
\end{center}
\end{figure}

Recall that $g$ was the homeomorphism of $(S^3, \Gamma)$ obtained by rotating $S^3$ about a $(1,1)$ curve on the torus $T$, followed by a reflection through a sphere meeting $T$ in two longitudes, and then a meridional rotation of $T$ by $6\pi/5$. In order to see what $g$ does to $\Gamma'$, we focus on the  square $\overline{1534}$ of $\Gamma'$.   Figure \ref{psi_stages} illustrates a rotation of the square $\overline{1534}$ about a diagonal, then a reflection of the square across a longitude.  The result of these two actions takes the projection of the knot $\overline{1534}$ to an identical projection. Thus after rotating the torus meridionally by $6\pi/5$, we see that $g$ takes $\Gamma'$ to itself inducing the automorphism $\psi=(2 4 3 1)$. It now follows that $\Z_5 \rtimes \Z_4 \leq \TSG(\Gamma') \leq S_5$. 

\begin{figure}[h!]
\begin{center}
\includegraphics[height=3.9cm]{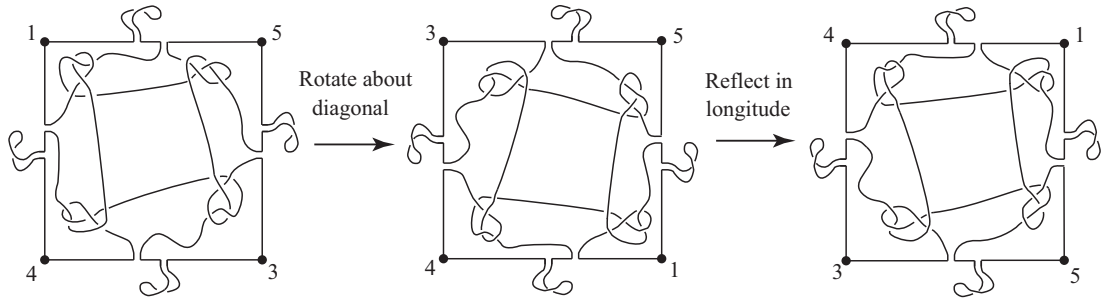}
\caption{Effect of $g$ on the square $\overline{1534}$}
\label{psi_stages}
\end{center}
\end{figure}

In order to prove that $\TSG(\Gamma') \cong \Z_5 \rtimes \Z_4$, we need to show $\TSG(\Gamma') \not \cong  \S_5$.  We  prove this by showing that the automorphism $(15)$ cannot be induced by a homeomorphism of $(S^3, \Gamma')$. 

From Figure \ref{psi_stages} we see that the square $\overline{1534}$ is the knot $4_1 \# 4_1 \# 4_1 \# 4_1$.  In order to see what would happen to this knot if the transposition $(15)$ were induced by a homeomorphism of $(S^3, \Gamma')$, we consider the square $\overline{5134}$. In Figures \ref{the_square_5134_p1} and  \ref{the_square_5134_p2} we isotop $\overline{5134}$ to a projection with only 10 crossings.  This means that $\overline{5134}$ cannot be the knot $4_1 \# 4_1 \# 4_1 \# 4_1$.  It follows that the automorphism $(15)$ cannot be induced by a homeomorphism of $(S^3, \Gamma')$.  Hence $\TSG(\Gamma')\not \cong S_5$.  However, the only subgroup of $\S_5$ that contains $\Z_5 \rtimes \Z_4$ and is not $S_5$ is the group $\Z_5 \rtimes \Z_4$. Thus in fact $\TSG(\Gamma') \cong \Z_5 \rtimes \Z_4$.

\begin{figure}[h!]
\begin{center}
\includegraphics[height=4.5cm]{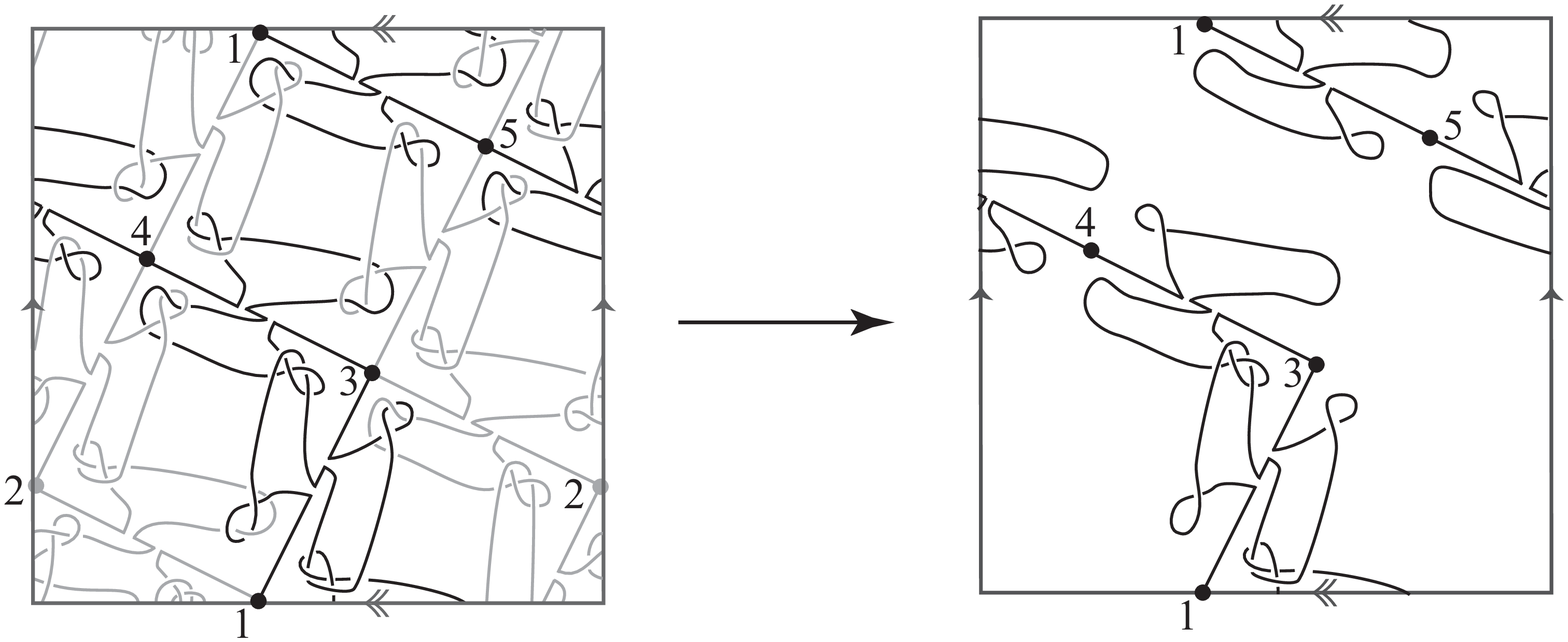}
\caption{The knot $\overline{5134}$ projected on the torus.}
\label{the_square_5134_p1}
\end{center}
\end{figure}

\begin{figure}[h!]
\begin{center}
\includegraphics[height=4.5cm]{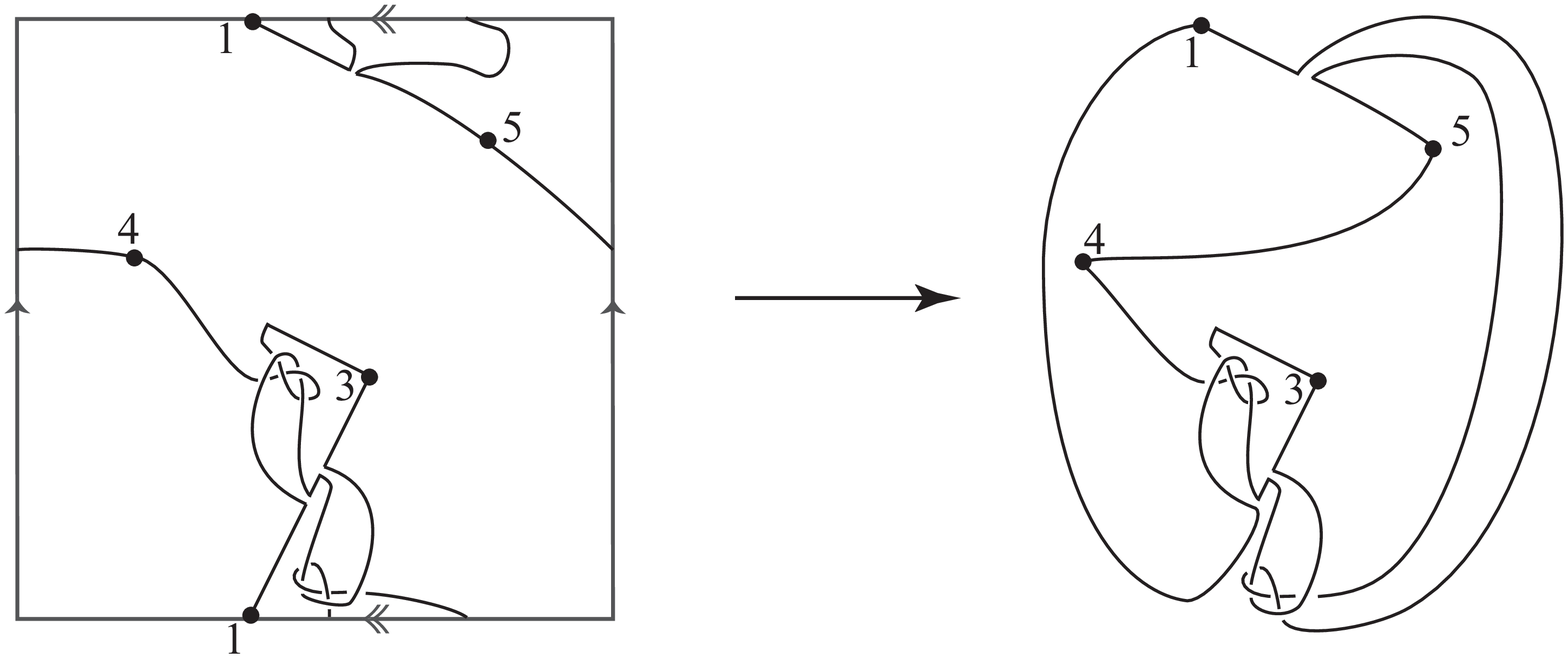}
\caption{A projection of $\overline{5134}$ on the torus after an isotopy, followed by a projection of $\overline{5134}$ on a plane.}
\label{the_square_5134_p2}
\end{center}
\end{figure}

 Thus every subgroup of $\mathrm{Aut}(K_5)$ is realizable for $K_5$.  Table \ref{table4} summarizes our results for $\TSG(K_5)$. 

\begin{table}[h!]
\centering
\begin{tabular}{| l | l | p{8cm} |}
\hline
{\bf Subgroup} & {\bf Realizable} & {\bf Reason} \\ \hline
$\S_5$ & Yes & By Figure \ref{S5}  \\ \hline
$A_5$ & Yes & Positively realizable \\ \hline
$\S_4$ & Yes & By modifying Figure \ref{S5}  \\ \hline
$A_4$ & Yes & Positively realizable \\ \hline
$\D_6$ & Yes & By Figure \ref{K5D6} \\ \hline
$\D_5$ & Yes & Positively realizable \\ \hline
$\D_4$ & Yes & By Figure \ref{K5D4} \\ \hline
$\D_3$ & Yes & Positively realizable \\ \hline
$\D_2$ & Yes & Positively realizable \\ \hline
$\Z_6$ & Yes & By modifying Figure \ref{K5D6} \\ \hline
$\Z_5 \rtimes \Z_4$ & Yes & By Figure \ref{TorFig8}  \\ \hline
$\Z_5$ & Yes & Positively realizable \\ \hline
$\Z_4$ & Yes & By modifying Figure \ref{K5D4}  \\ \hline
$\Z_3$ & Yes & Positively realizable \\ \hline
$\Z_2$ & Yes & Positively realizable \\ \hline
\hline
\end{tabular}

\caption{Non-trivial realizable groups for $K_5$}
\label{table4}
\end{table}

\bigskip

\section{\bf{Topological Symmetry Groups of $\mathbf{K_6}$}}

The following is a complete list of all the subgroups of $\Aut(K_6) \cong \S_6$:

\noindent $\S_6$, $A_6$, $\S_5$, $A_5$, $S_2[S_3]$\footnote[1]{$B[A]$ represents a wreath product of $A$ by $B$.}, $\S_4 \times \Z_2$, $A_4 \times \Z_2$, $\S_4$, $A_4$, $\Z_5 \rtimes \Z_4$, $\D_3 \times \D_3$, $(\Z_3 \times \Z_3) \rtimes \Z_4$, $(\Z_3 \times \Z_3) \rtimes \Z_2$, $\D_3 \times \Z_3$, $\Z_3 \times \Z_3$, $\D_6$, $\D_5$, $\D_4$, $\D_4 \times \Z_2$, $\D_3$, $\D_2$, $\Z_6$, $\Z_5$, $\Z_4$, $\Z_4 \times \Z_2$, $\Z_3$, $\Z_2$, $\Z_2 \times \Z_2 \times \Z_2$  (see \cite{gotz6} and independently verified using the program GAP).

\begin{figure}[h]
\begin{center}
\includegraphics[height=5.2cm]{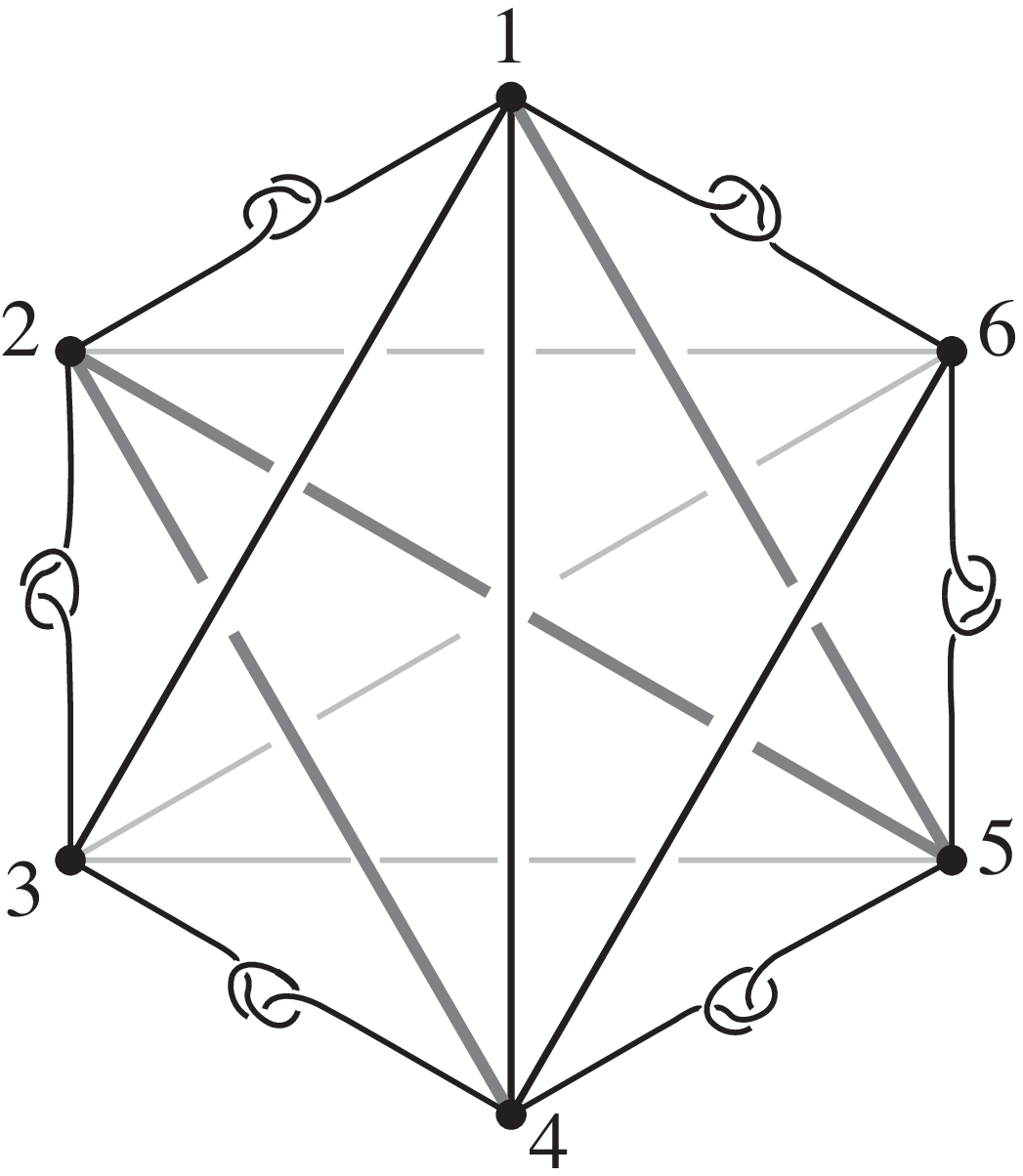}
\caption{$\TSG_+(\Gamma) \cong \D_6$.}
\label{budshade}
\end{center}
\end{figure}

Consider the embedding $\Gamma$ of $K_6$ illustrated in Figure \ref{budshade}. There are three paths in $\Gamma$ on different levels that look like the letter ``Z" which are highlighted in Figure~\ref{budshade}.  The top Z-path is $\overline{3146}$, the middle Z-path is $\overline{4251}$, and the bottom Z-path is $\overline{5362}$. The knotted cycle $\overline{123456}$ must be setwise invariant under every homeomorphism of $\Gamma$, and hence $\TSG_+(\Gamma)\leq D_6$. The automorphism $(123456)$ is induced by a glide rotation that cyclically permutes the Z-paths. Consider the homeomorphism obtained by rotating $\Gamma$ by $180^\circ$ about the line through vertices 2 and 5 and then pulling the edges $\overline{13}$ and $\overline{46}$ to the top level while pushing the lower edges down. The result of this homeomorphism is that the top Z-path  $\overline{3146}$ goes to the top Z-path $\overline{1364}$, the middle Z-path $\overline{4251}$ goes to to middle Z-path $\overline{6253}$, and the bottom Z-path $\overline{5362}$ goes to the bottom Z-path $\overline{5142}$.  Thus the homeomorphism leaves $\Gamma$ setwise invariant inducing the automorphism $(13)(46)$.  It follows that $\TSG_+(\Gamma)=\langle (123456),(13)(46) \rangle \cong \D_6$.  Finally, since the edge $\overline{12}$ is not pointwise fixed by any non-trivial element of $\TSG_+(\Gamma)$, by the Subgroup Theorem the groups $\Z_6$, $\D_3$, $\Z_3$, $\D_2$ and $\Z_2$ are positively realizable for $K_6$.

Consider the embedding, $\Gamma$ of $K_6$ illustrated in Figure \ref{K6D5} with vertex 6 at infinity. The automorphisms $(13524)$ and $(25)(34)$ are induced by rotations.  Also since $\overline{13524}$ is the only 5-cycle which is knotted, $\overline{13524}$ is setwise invariant under every homeomorphism of $(S^3, \Gamma)$. Hence $\TSG_+(\Gamma) \cong \D_5$. Also since $\overline{15}$ is not pointwise fixed under any homeomorphism, by the Subgroup Theorem, $\Z_5$ is positively realizable for $K_6$.

\begin{figure}[h]
\begin{center}
\includegraphics[height=4cm]{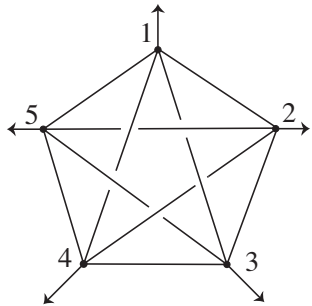}
\caption{$\TSG_+(\Gamma) \cong \D_5$.}
\label{K6D5}
\end{center}
\end{figure}

Next consider the embedding, $\Gamma$ of $K_6$ illustrated in Figure \ref{D3}. The automorphisms $(1 2 3)(4 5 6)$ and $(1 2 3)(4 6 5)$ are induced by glide rotations and $(4 5)(12)$ is induced by turning the figure upside down.  Also if we consider the circles $\overline{123}$ and $\overline{465}$ as cores of complementary solid tori, then $(14)(25)(36)$ is induced by an orientation preserving homeomorphism that switches the two solid tori.

\begin{figure}[h]
\begin{center}
\includegraphics[height=5.2cm]{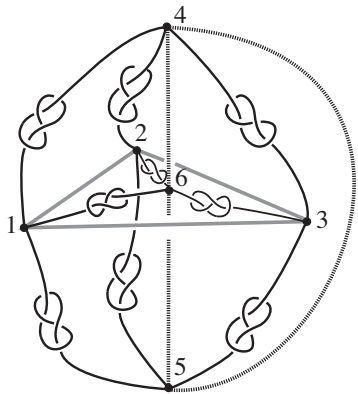}
\caption{$\TSG_+(\Gamma) \cong \D_3 \times \D_3$.}
\label{D3}
\end{center}
\end{figure}

Observe that every homeomorphism of $(S^3, \Gamma)$ takes the pair of triangles $\overline{123}\cup\overline{456}$ to itself, since this is the only pair of complementary triangles not containing knots.  The automorphism group of the union of two triangles is $S_2[S_3]$ \cite{Fr}. Thus $\TSG_+(\Gamma)\leq S_2[S_3]$.  Note that the transpositions $(12)$ and $(45)$ are each induced by a reflection followed by an isotopy. Thus $\TSG(\Gamma) \cong S_2[S_3]$, since $(123)(456)$, $(123)(465)$, $(12)$ and $(14)(25)(36)$ generate $S_2[S_3]$. However, by the Complete Graph Theorem, $\TSG_+(\Gamma) \not\cong S_2[S_3]$.  Thus $\TSG_+(\Gamma)$ must be an index 2 subgroup of $S_2[S_3]$ containing $f=(123)(456)$, $g=(123)(465)$, $\phi = (45)(12)$ and $\psi = (14)(25)(36)$.  Observe that $\phi\psi$ is the involution $(42)(51)(36)$, and $f$ commutes with $\psi$ and also $f \phi \psi = \phi \psi f^{-1}$, while $g$ commutes with $\phi \psi$ and $g \psi =\psi g^{-1}$. Thus $S_2[S_3] \gneq \TSG_+(\Gamma) \geq \langle f,\phi\psi\rangle\times \langle g, \psi\rangle \cong \D_3 \times \D_3$.  It follows that $\TSG_+(\Gamma)\cong  \D_3 \times \D_3$.

The subgroup $\langle f, g, \psi \rangle$ is isomorphic to $\D_3 \times \Z_3$ because $\psi$ commutes with $f$ and $g\psi = \psi g^{-1}$.  We add the non-invertible knot $8_{17}$ to every edge of the triangles $\overline{123}$ and $\overline{456}$ to obtain an embedding $\Gamma_1$. Now the automorphism $\phi = (45)(12)$ cannot be induced by an orientation preserving homeomorphism of $(S^3, \Gamma_1)$.  However, $f$, $g$, and $\psi$ are still induced by orientation preserving homeomorphisms.  Thus $\TSG_+(\Gamma_1) \cong \D_3\times \Z_3$ since $\D_3\times \Z_3$ is a maximal subgroup of $\D_3 \times \D_3$.

Also $\langle f, g, \phi \rangle$ is isomorphic to $ (\Z_3 \times \Z_3) \rtimes \Z_2$ because $f\phi =\phi f^{-1}$ and $g\phi = \phi g^{-1}$. Again starting with $\Gamma$ in Figure \ref{D3}, we place  $5_2$ knots on the edges of the triangle $\overline{123}$ so that $\psi$ is no longer induced. Thus creating and embedding $\Gamma_2$ with $\TSG_+(\Gamma_2) \cong (\Z_3 \times \Z_3) \rtimes \Z_2$ since $(\Z_3 \times \Z_3) \rtimes \Z_2$ is a maximal subgroup of $\D_3 \times \D_3$.

Finally $\langle f,g \rangle$ is isomorphic to $ \Z_3 \times \Z_3$. If we place equivalent non-invertible knots on each edge of the triangle $\overline{123}$ and a another set (distinct from the first set) of equivalent non-invertible knots on each edge of $\overline{456}$ we obtain an embedding $\Gamma_3$ with $\TSG_+(\Gamma_3) \cong \Z_3\times \Z_3$ since $\Z_3 \times \Z_3$ is a maximal subgroup of $(\Z_3 \times \Z_3) \rtimes \Z_2$.

\medskip

We summarize our results on positively realizable for $K_6$ in Table \ref{table5}.

\begin{table}[h!]
\centering
\begin{tabular}{| l | l | p{6cm} |}
\hline
{\bf Subgroup} & {\bf Positively Realizable} & {\bf Reason} \\ \hline
$A_5$ & No & By $A_5$ Theorem \\ \hline
$\S_4$ & No & By $\S_4$ Theorem \\ \hline
$A_4$ & No & By $A_4$ Theorem \\ \hline
$\D_6$ & Yes & By Figure \ref{budshade}  \\ \hline
$\D_5$ & Yes & By Figure \ref{K6D5}\\ \hline
$\D_4$ & No & By Lemma 2\\ \hline
$\D_3 \times \D_3$ & Yes & By Figure \ref{D3} \\ \hline
$\D_3 \times \Z_3$ & Yes & By modifying Figure \ref{D3} \\ \hline
$\D_3$ & Yes & By Subgroup Theorem \\ \hline
$\D_2$ & Yes & By Subgroup Theorem \\ \hline
$\Z_6$ & Yes & By Subgroup Theorem \\ \hline
$\Z_5$ & Yes & By Subgroup Theorem \\ \hline
$\Z_4$ & No & By Lemma 2 \\ \hline
$(\Z_3 \times \Z_3) \rtimes \Z_2$ & Yes & By modifying Figure \ref{D3} \\ \hline
$\Z_3 \times \Z_3$ & Yes & By modifying Figure \ref{D3} \\ \hline
$\Z_3$ & Yes & By Subgroup Theorem \\ \hline
$\Z_2$ & Yes & By Subgroup Theorem \\ \hline
\hline
\end{tabular}

\caption{Non-trivial positively realizable groups for $K_6$}
\label{table5}
\end{table}

By adding appropriate equivalent chiral knots to each edge, every group which is positively realizable for $K_6$ is also realizable for $K_6$.  Thus we only need to determine reliability for the groups $\S_6$, $A_6$, $\S_5$, $A_5$, $\S_4 \times \Z_2$, $A_4 \times \Z_2$, $\S_4$, $A_4$, $\Z_5 \rtimes \Z_4$, $(\Z_3 \times \Z_3) \rtimes \Z_4$, $\D_4$, $\D_4 \times \Z_2$, $\Z_4$, $\Z_4 \times \Z_2$, and $\Z_2 \times \Z_2 \times \Z_2$. Note that in Figure~\ref{D3} we already determined that $S_2[S_3]$ is realizable for $K_6$.

Let $\Gamma_4$ be the embedding of $K_6$ illustrated in Figure \ref{D3} with a left handed trefoil added to each edge of $\overline{123}$ and a right handed trefoil added to each edge of $\overline{456}$. The pair of triangles are setwise invariant since no other edges contain trefoils.  Both $(123)(456)$ and  $(123)(465)$ are induced by homeomorphisms of $(\Gamma_4, S^3)$.  Also if we reflect in the plane containing vertices 4, 5, 6, and 1 then all the trefoils switch from left-handed to right-handed and vice versa.  If we then interchange the complementary solid tori which have the triangles as cores followed by an isotopy, we obtain an orientation reversing homeomorphism that induces the order 4 automorphism $(14)(25)(36)(23) =(14)(2536)$. Now $\langle (14)(2536),(123)(456),(123)(465)\rangle\cong (\Z_3 \times \Z_3) \rtimes \Z_4$.  

We see as follows that $\TSG(\Gamma_4)$ cannot be larger than $(\Z_3 \times \Z_3) \rtimes \Z_4$. Suppose that the automorphism $(1 2)$ is induced by a homeomorphism $f$.  By Lemma 1, $f$ must be orientation reversing. But $f(\overline{456}) = \overline{456}$, which is impossible because $\overline{456}$ contains only right handed trefoils.  Thus $ \TSG(\Gamma_4)\not \cong S_2[S_3]$.  Note that the only proper subgroup of $S_2[S_3]$ containing $(\Z_3 \times \Z_3) \rtimes \Z_4$ is $(\Z_3 \times \Z_3) \rtimes \Z_4$.  Thus $\TSG(\Gamma_4) \cong (\Z_3 \times \Z_3) \rtimes \Z_4$.

Now let $\Gamma$ be the embedding of $K_6$ illustrated in Figure \ref{D4}. Observe that the linking number $\lk(\overline{135}, \overline{246}) = \pm 1$, but $\lk(\overline{136}, \overline{245}) = 0$. Thus the automorphism $(56)$ cannot be induced by a homeomorphism of $(S^3, \Gamma)$. Since every homeomorphism of $(S^3, \Gamma)$ takes $\overline{1234}$ to itself, it follows that $\TSG(\Gamma) \leq \D_4$.  The automorphism $(1 2 3 4)(5 6)$ is induced by a rotation followed by a reflection and an isotopy.  In addition the automorphism $(1 4)(2 3)(5 6)$ is induced by turning the figure upside down. Thus $\TSG(\Gamma) \cong \D_4$ generated by the automorphisms $(1 2 3 4)(5 6)$ and $(1 4)(2 3)(5 6)$.

\begin{figure}[h!]
\begin{center}
\includegraphics[height=5cm]{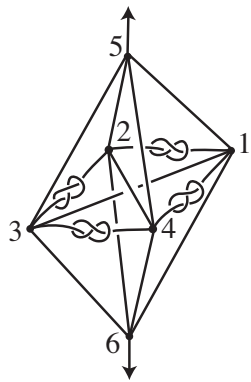}
\caption{$\TSG(\Gamma) \cong \D_4$.}
\label{D4}
\end{center}
\end{figure}

Now let $\Gamma'$ be obtained from Figure \ref{D4} by replacing the knot $4_1$ with the non-invertible and positively achiral knot $12_{427}$. Then the square $\overline{1234}$ can no longer be inverted. In this case $(1 2 3 4)(5 6)$ generates $\TSG(\Gamma')$ and thus $\TSG(\Gamma') \cong \Z_4$.
\medskip

For the next few groups we will use the following lemma.

\begin{N4Cycle}
\cite{MOB} For any embedding $\Gamma$ of $K_6$ in $S^3$, and any labelling of the vertices of $K_6$ by the numbers 1 through 6, there is no homeomorphism of $(S^3, \Gamma)$ which induces the automorphism $(1234)$.
\end{N4Cycle}

Consider the subgroup $\Z_5 \rtimes \Z_4\leq \mathrm{Aut}(K_6)$.  The presentation of $\Z_5 \rtimes \Z_4$ as a subgroup of $S_6$ gives the relation $x^{-1}yx = y^2$ for some elements $x$, $y\in\Z_5 \rtimes \Z_4$ of orders 4 and 5 respectively.  Suppose that for some embedding $\Gamma$ of $K_6$, we have $\TSG(\Gamma) \cong \Z_5 \rtimes \Z_4$.  Without loss of generality, we can assume that  $y=(1 2 3 4 5)$ satisfies the relation $x^{-1}yx = y^2$.  By the 4-Cycle Theorem, any order 4 element of $\TSG(\Gamma)$ must be of the form $x=(a b c d)(e f)$.  However, there is no element in $\mathrm{Aut}(K_6)$ of the form $x=(a b c d)(e f)$ that together with $y=(12345)$ satisfies this relation. Thus there can be no embedding $\Gamma$ of $K_6$ in $S^3$ such that $\TSG(\Gamma) \cong \Z_5 \rtimes \Z_4$.

Now consider the subgroup $\Z_4 \times \Z_2\leq \mathrm{Aut}(K_6)$. By the 4-Cycle Theorem, without loss of generality we may assume that if $\TSG(\Gamma)$ contains an element of order 4 for some embedding $\Gamma$ of $K_6$, then  $\TSG(\Gamma)$ contains the element $(1234)(56)$. Computation shows that the only transposition in $ \mathrm{Aut}(K_6)$ that commutes with $(1234)(56)$ is $(56)$, which cannot be an element of $\TSG(\Gamma)$ since this would imply that $(1234)$ is an element of $\TSG(\Gamma)$.  Furthermore the only order 2 element of  $\mathrm{Aut}(K_6)$  that commutes with $(1234)(56)$ and is not a transposition is $(13)(24)$, which is already in the group generated by $(1234)(56)$. Thus there is no embedding $\Gamma$ of $K_6$ in $S^3$ such that $\TSG(\Gamma)$ contains the group $\Z_4 \times \Z_2$.  This rules out all of the groups $\S_4 \times \Z_2$, $\D_4 \times \Z_2$ and $\Z_4 \times \Z_2$ as possible topological symmetry groups for embeddings of $K_6$ in $S^3$.
\medskip

For the group $\Z_2 \times \Z_2 \times \Z_2$ we will use the following result.

\begin{CnG} \cite{ConGor} For any embedding $\Gamma$ of $K_6$ in $S^3$, the mod 2 sum of the linking numbers of all pairs of complementary triangles in $\Gamma$ is 1. 
\end{CnG}

Now suppose that for some embedding $\Gamma$ of $K_6$ in $S^3$ we have $\TSG(\Gamma) \cong  \Z_2 \times \Z_2 \times \Z_2$. It can be shown that the subgroup $\Z_2 \times \Z_2 \times \Z_2\leq \mathrm{Aut}(K_6)$ contains three disjoint transpositions.  Without loss of generality we can assume that $\TSG(\Gamma)$ contains $(13)$, $(24)$, and $(56)$, which are induced by homeomorphisms $h$, $f$, and $g$ of $(S^3,\Gamma)$ respectively. Since any three vertices of $\Gamma$ determine a pair of disjoint triangles, we can use a triple of vertices to represent a pair of disjoint triangles. For example, we use the triple $123$ to denote the pair of triangles $\overline{123}$ and $\overline{456}$. With this notation, the orbits of the ten pairs of disjoint triangles in $K_6$ under the group $\langle (13), (24), (56) \rangle$ are:

\medskip

\centerline{$\{ 123, 143 \}$, $\{ 124, 324 \}$, $\{ 125, 325, 145, 126 \}$, $\{ 135, 136 \}$}

\medskip

Since $h$, $f$, and $g$ are homeomorphisms of $(S^3, \Gamma)$ the links in a given orbit all have the same (mod 2) linking number.  Since each of these orbits has an even number of pair of triangles, this contradicts Conway Gordon.  Thus $\Z_2 \times \Z_2 \times \Z_2\not \cong \TSG(\Gamma)$.  Hence $\Z_2 \times \Z_2 \times \Z_2$ is not realizable for $K_6$

\medskip

Table \ref{table6} summarizes our realizability results for $K_6$. Recall that for $n=4$ and $n=5$ every subgroup of $S_n$ is realizable for $K_n$.  However, as we see from Table \ref{table6}, this is not true for $n=6$.

\begin{table}[h!]
\centering
\begin{tabular}{| l | l | p{9.5cm} |}
\hline
{\bf Subgroup} & {\bf Realizable} & {\bf Reason} \\ \hline
$\S_6$ & No & $\TSG_+(K_6)$ cannot be $\S_6$ or $A_6$  \\ \hline
$A_6$ & No & $\TSG_+(K_6)$ cannot be $A_6$ \\ \hline
$\S_5$ & No & $\TSG_+(K_6)$ cannot be $\S_5$ or $A_5$ \\ \hline
$A_5$ & No & $\TSG_+(K_6)$ cannot be $A_5$ \\ \hline
$\S_4 \times \Z_2$ & No & $\TSG_+(K_6)$ cannot be $\S_4 \times \Z_2$ or $\S_4$ \\ \hline
$\S_4$ & No & $\TSG_+(K_6)$ cannot be $\S_4$ or $A_4$ \\ \hline
$A_4 \times \Z_2$ & No & $\TSG_+(K_6)$ cannot be $A_4 \times \Z_2$ or $A_4$ \\ \hline
$A_4$ & No & $\TSG_+(K_6)$ cannot be $A_4$ \\ \hline
$\D_6$ & Yes & Positively realizable \\ \hline
$\D_5$ & Yes & Positively realizable \\ \hline
$\D_4 \times \Z_2$ & No & $\TSG_+(K_6)$ cannot be $\D_4 \times \Z_2$, $\D_4$, $\Z_4 \times \Z_2$, $\Z_2 \times \Z_2 \times \Z_2$  \\ \hline
$\D_4$ & Yes & By Figure \ref{D4}  \\ \hline
$S_2[S_3]$ & Yes & By Figure \ref{D3} \\ \hline
$\D_3 \times \D_3$ & Yes & Positively realizable \\ \hline
$\D_3 \times \Z_3$ & Yes & Positively realizable \\ \hline
$\D_3$ & Yes & Positively realizable \\ \hline
$\D_2$ & Yes & Positively realizable \\ \hline
$\Z_6$ & Yes & Positively realizable \\ \hline
$\Z_5 \rtimes \Z_4$ & No & By 4-Cycle Theorem\\ \hline
$\Z_5$ & Yes & Positively realizable \\ \hline
$\Z_4 \times \Z_2$ & No &  By 4-Cycle Theorem \\ \hline
$\Z_4$ & Yes & By modifying Figure \ref{D4}\\ \hline
$(\Z_3 \times \Z_3) \rtimes \Z_4$ & Yes & By modifying Figure \ref{D3}\\ \hline
$(\Z_3 \times \Z_3) \rtimes \Z_2$ & Yes & Positively realizable \\ \hline
$\Z_3 \times \Z_3$ & Yes & Positively realizable \\ \hline
$\Z_3$ & Yes & Positively realizable \\ \hline
$\Z_2 \times \Z_2 \times \Z_2$ & No & By  Conway Gordon Theorem \\ \hline
$\Z_2$ & Yes & Positively realizable \\ \hline
\hline
\end{tabular}

\caption{Non-trivial realizable groups for $K_6$}
\label{table6}
\end{table}





\newpage

\acknowledgements{Acknowledgements}

The first author would like to thank Claremont Graduate University for its support while he pursued the study of Topological Symmetry Groups for his Ph.D thesis.
The second author would like to thank the Institute for Mathematics and its Applications at the University of Minnesota for its hospitality while she was a long term visitor in the fall of 2013.

\bibliographystyle{mdpi}
\makeatletter
\renewcommand\@biblabel[1]{#1. }
\makeatother


\end{document}